\documentclass[reqno,tbtags,intlimits,a4paper,oneside,12pt]{amsart}
\usepackage[cp1251]{inputenc}%
\usepackage[english]{babel}
\usepackage{amssymb,upref,calc,url}
\usepackage[mathscr]{eucal}
\usepackage{graphicx}
\usepackage{amsmath,amscd,amsthm,verbatim}
\usepackage[all]{xy}
\usepackage[in]{fullpage}
\newtheorem{thm}{Theorem}[section]
\newtheorem{lem}[thm]{Lemma}
\newtheorem{cor}[thm]{Corollary}

\newtheorem{prb}[thm]{Problem}
\newtheorem{prop}[thm]{Properties}

\theoremstyle{definition}

\newtheorem{rem}[thm]{Remark}
\DeclareMathOperator{\Der}{Der}

\begin{document}

\title{Differentiation of Genus 3 Hyperelliptic Functions}
\author{Elena~Yu.~Bunkova}
\address{Steklov Mathematical Institute of Russian Academy of Sciences, Moscow, Russia.}
\email{bunkova@mi.ras.ru}
\thanks{Supported in part by Young Russian Mathematics award, Royal Society International Exchange grant and the RFBR project 17-01-00366 A}

\begin{abstract}
In this work we give an explicit solution to the problem of differentiation of hyperelliptic functions in genus $3$ case.
It is a genus $3$ analogue of the result of F.~G.~Frobenius and L.~Stickelberger \cite{FS}.

Our method is based on the series of works by V.~M.~Buchstaber, D.~V.~Leikin and V.~Z.~Enolskii \cite{BL, B, BEL, BPol}.
We describe a polynomial map $p\colon \mathbb{C}^{3g} \to \mathbb{C}^{2g}$.
For $g = 1,2,3$ we describe $3g$ polynomial vector fields in $\mathbb{C}^{3g}$ projectable for $p$ and their polynomial Lie algebras.
We obtain the corresponding derivations of the field of hyperelliptic functions.
\end{abstract}

\maketitle

\section{Introduction}

In \cite{BL} the problem of differentiation of Abelian functions was described.
It has deep relations \cite{B} with KdV equations theory.

An \emph{Abelian function} is a meromorphic function on $\mathbb{C}^g$ with a lattice of periods $\Gamma \subset \mathbb{C}^g$
of rank $2g$. We say that an Abelian function is a meromorphic function on the complex torus $T^g = \mathbb{C}^g/\Gamma$.

Let us consider hyperelliptic curves of genus $g$ in the model
\begin{equation} \label{1}
\mathcal{V}_\lambda = \{(X,Y)\in\mathbb{C}^2 \colon
Y^2 = X^{2g+1} + \lambda_4 X^{2 g - 1}  + \lambda_6 X^{2 g - 2} + \ldots + \lambda_{4 g} X + \lambda_{4 g + 2}\}. 
\end{equation}
Such a curve depends on the parameters $\lambda = (\lambda_4, \lambda_6, \ldots, \lambda_{4 g}, \lambda_{4 g + 2}) \in \mathbb{C}^{2 g}$.

Denote by $\mathcal{B} \subset \mathbb{C}^{2g}$ the subspace of parameters such that $\mathcal{V}_{\lambda}$ is non-singular for $\lambda \in \mathcal{B}$.
We have $\mathcal{B} = \mathbb{C}^{2g} \backslash \Sigma$ where $\Sigma$ is the discriminant curve.

In this work we will deal only with Abelian functions in the model \eqref{1}. In \cite{B} such functions are called hyperelliptic.

A hyperelliptic function of genus $g$ is 
 a smooth function defined on an open dense subset of $\mathbb{C}^g \times \mathcal{B}$,
such that for each $\lambda \in \mathcal{B}$ it's restriction to $\mathbb{C}^g \times \lambda$
 is Abelian with $T^g$ the Jacobian $\mathcal{J}_\lambda$ of $\mathcal{V}_\lambda$.
 
Let $\mathcal{U}$ be the space of the fiber bundle $\pi: \mathcal{U} \to \mathcal{B}$ with fiber over $\lambda \in \mathcal{B}$ the Jacobian~$\mathcal{J}_\lambda$
of the curve $\mathcal{V}_\lambda$.
Thus, a hyperelliptic function is a fiberwise meromorphic function on~$\mathcal{U}$.
According to Dubrovin--Novikov theorem \cite{DN}, the space $\mathcal{U}$ is birationally equivalent to the complex linear space $\mathbb{C}^{3g}$.
Denote the field of hyperelliptic functions by $\mathcal{F}$.

The general statement of the Problem of Differentiation of Abelian Functions is given in~\cite{BL}.
In this paper we consider the special case of the model \eqref{1}:

\begin{prb}[Problem of Differentiation of Hyperelliptic Functions]\text{ } \label{p1}

\begin{enumerate}
 \item Find the $3g$ generators of the $\mathcal{F}$-module $\Der \mathcal{F}$ of derivations of the field $\mathcal{F}$.
 \item Describe the structure of Lie algebra $\Der \mathcal{F}$ (i.e. find the commutation relations).
\end{enumerate}
\end{prb}

In \cite{BL} the Problem of Differentiation of Abelian Functions is solved.
Yet this solution is not explicit in the sense that it gives general methods that lead to increasing
amount of problems and calculations to give explicit answers for each genera. 

\eject

In this work we introduce a different approach. The main idea is taken from \cite{B}.
It consists of replacing the bundle $\pi: \mathcal{U} \to \mathcal{B}$
by a polynomial map $p: \mathbb{C}^{3g} \to \mathbb{C}^{2g}$. 
This relies on a theorem from \cite{BEL}.
We give polynomial vector fields in $\mathbb{C}^{3g}$ projectable for $p$ and their polynomial Lie algebras \cite{BPol}.
By construction this solves Problem \ref{p1} for $g=1,2,3$.

The classical genus 1 case formulas are given in \cite{BL}.
The genus 2 case formulas were first obtained in \cite{B}.
This work gives the full explicit answer for Problem \ref{p1} in genus 3 case.
Our method should also work to solve Problem \ref{p1} for genus $g > 3$.

In this work we use the theory of hyperelliptic Kleinian functions (see \cite{BEL, Baker, BEL-97, BEL-12}, and~\cite{WW} for elliptic functions).
Let $\sigma(u, \lambda)$ be the hyperelliptic sigma function (or elliptic sigma function in genus $g=1$ case).
We use the notation
\begin{equation} \label{n}
\zeta_{k} = \frac{\partial}{\partial_{u_k}} \ln \sigma(u, \lambda), \qquad
\wp_{i; k_1, \ldots, k_n} = - \frac{\partial^{i + n}}{\partial_{u_1}^{i} \partial_{u_{k_1}} \cdots \partial_{u_{k_n}}} \ln \sigma(u, \lambda),
\end{equation}
where $u = (u_1, u_3, \ldots, u_{2 g -1})$, $n \geqslant 0$, $i + n \geqslant 2$, $k_s \in \{ 1, 3, \ldots, 2 g - 1\}$. In the case $n = 0$ we will skip the semicolon.
Note that our notation for the variables $u_k$ differs from the one in \cite{BEL, BEL-12} as $u_{i} \leftrightarrow u_{2g + 1 - 2 i}$.
This follows \cite{B} and gives us a supplementary homogeneity for the grading of variables corresponding to their indices.

The paper is organized as follows: 
In Sections \ref{s2}--\ref{s4} we deal with the bundle $\pi:\mathcal{U} \to \mathcal{B}$.
In Section \ref{s5} we describe the polynomial map $p: \mathbb{C}^{3g} \to \mathbb{C}^{2g}$ and it's relation to $\pi: \mathcal{U} \to \mathcal{B}$.
In Sections \ref{s6}--\ref{s9} we give polynomial vector fields in $\mathbb{C}^{3g}$ projectable for $p$.
In Section \ref{s10} we use these vector fields to get the generators of $\Der \mathcal{F}$ and their Lie algebra.

In Section \ref{s2} we investigate Problem \ref{p1}, in Section \ref{s3} we give it's answers \cite{FS, B} \mbox{for $g = 1,2$.}
In Section \ref{s4} we give $2g$ polynomial vector fields on $\mathcal{B}$ with some additional properties we will mention.

In Section \ref{s6} we relate Problem \ref{p1} to a problem of constructing polynomial vector fields in $\mathbb{C}^{3g}$. 
In Sections \ref{s7}--\ref{s9} we give $3g$ polynomial vector fields in $\mathbb{C}^{3g}$ for $g = 1,2$ and~$3$,
we describe their properties and their polynomial Lie algebras.

In Section \ref{s10} we give the theorem solving Problem \ref{p1} in genus $3$ case.

The author thanks V. M. Buchstaber for fruitful discussions of the results.

\vspace{-1mm}

\section{Problem of Differentiation over Parameters} \label{s2}

Let us consider the Problem of Differentiation of Hyperelliptic Functions \ref{p1}.

In $\mathbb{C}^g \times \mathcal{B} \subset \mathbb{C}^{3g}$ we take the coordinates
$(u_1, u_3, \ldots, u_{2 g -1}, \lambda_4, \lambda_6, \ldots, \lambda_{4 g}, \lambda_{4 g + 2})$.
Their~indices correspond to their grading. Further most functions considered will be homogeneous with respect to this grading.
By definition, hyperelliptic functions are meromorphic functions in $u = (u_1, u_3, \ldots, u_{2 g -1})$ with $2g$ periods that depend on the
\mbox{parameters $\lambda = (\lambda_4, \lambda_6, \ldots, \lambda_{4 g}, \lambda_{4 g + 2})$.}

If $f(u, \lambda)$ is a hyperelliptic function, then
${\partial \over \partial u_1} f(u, \lambda)$, ${\partial \over \partial u_3} f(u, \lambda)$, $\ldots$, ${\partial \over \partial u_{2 g -1}} f(u, \lambda)$,
\mbox{are hyperelliptic} functions with the same periods. Therefore, we have ${\partial \over \partial u_k} \in \Der \mathcal{F}$ \mbox{for $k = 1$, $3$, $\ldots$, $2 g -1$.}
We denote ${\partial \over \partial u_k} = \mathscr{L}_k$ for odd $k$ from $1$ to $2 g - 1$.

On the other hand, if we take a hyperelliptic function as a smooth function defined on an open dense subset of $\mathbb{C}^g \times \mathcal{B}$,
it's derivative ${\partial \over \partial \lambda_k} f(u, \lambda)$ in general would not be hyperelliptic, thus
${\partial \over \partial \lambda_k} \notin \Der \mathcal{F}$.
Therefore, Problem \ref{p1} has the following subproblem: 

\begin{prb} \label{p2}
For a vector field $L$ in $\mathcal{B}$ find a vector field $\mathscr{L}$ in $\mathbb{C}^g \times \mathcal{B}$ \mbox{such that $\mathscr{L}\in  \Der \mathcal{F}$,}
the vector field $\mathscr{L}$ is projectable for $\pi$ and $L$ is the pushforward of $\mathscr{L}$
\mbox{(i.e.~$\mathscr{L}(\pi^*(f)) = \pi^* L(f)$} for any function $f$ in $\mathcal{B}$).
\end{prb}

A solution to Problem \ref{p2} for $2g$ vector fields $L_{2k}$ independent at any point of $\mathcal{B}$
together with the vector fields $\mathscr{L}_{k}$ for $k = 1, 3, \ldots, 2 g -1$
would give a solution to Problem~\ref{p1} (1).

\section{Known solutions of the Problem of Differentiation of~Hyperelliptic~Functions} \label{s3}

\subsection{Genus 1} In the elliptic case the generators of the $\mathcal{F}$-module $\Der \mathcal{F}$ have been found in~\cite{FS}.
We give the generators and their Lie algebra:
\begin{align*}
\mathscr{L}_0 &= 4 \lambda_4 \partial_{\lambda_4} + 6 \lambda_6 \partial_{\lambda_6} - u_1 \partial_{u_1}, \rule[-10pt]{0pt}{16pt} & [\mathscr{L}_0, \mathscr{L}_1] &= \mathscr{L}_1,\\
\mathscr{L}_1 &= \partial_{u_1}, & [\mathscr{L}_0, \mathscr{L}_2] &= 2 \mathscr{L}_2, \\
\mathscr{L}_2 &= 6 \lambda_6 \partial_{\lambda_4} - {4 \over 3} \lambda_4^2 \partial_{\lambda_6} - \zeta_1 \partial_{u_1}, & 
[\mathscr{L}_1, \mathscr{L}_2] &= \wp_2 \mathscr{L}_1.
\end{align*}

\subsection{Genus 2} \label{ss2} In this case Problem \ref{p1} was solved in~\cite{B}.
The generators are (see~\cite{B}, Theorem 29):
\begin{align*}
\mathscr{L}_0 &= L_0 - u_1 \partial_{u_1} - 3 u_3 \partial_{u_3}, &
\mathscr{L}_2 &= L_2 + \left(- \zeta_1 + {4 \over 5} \lambda_4 u_3\right) \partial_{u_1} - u_1 \partial_{u_3}, \\
\mathscr{L}_1 &= \partial_{u_1}, &
\mathscr{L}_4 &= L_4 + \left(- \zeta_3 + {6 \over 5} \lambda_6 u_3\right) \partial_{u_1} - \left(\zeta_1 + \lambda_4 u_3\right) \partial_{u_3}, \\
\mathscr{L}_3 &= \partial_{u_3}, &
\mathscr{L}_6 &= L_6 + {3 \over 5} \lambda_8 u_3 \partial_{u_1} - \zeta_3 \partial_{u_3},
\end{align*}
where the vector fields $L_k$ on $\mathcal{B}$ will be given explicitly in Section \ref{s4}. 

The Lie algebra is (see \cite{B}, Theorem 32):
\begin{align*}
[\mathscr{L}_0, \mathscr{L}_k] &= k \mathscr{L}_k, \quad k=1,2,3,4,6, \qquad \qquad \qquad [\mathscr{L}_1, \mathscr{L}_3] = 0, \\
[\mathscr{L}_1, \mathscr{L}_2] &= \wp_2 \mathscr{L}_1 - \mathscr{L}_3, \qquad
[\mathscr{L}_1, \mathscr{L}_4] = \wp_{1;3} \mathscr{L}_1 + \wp_2 \mathscr{L}_3, \qquad
[\mathscr{L}_1, \mathscr{L}_6] = \wp_{1;3} \mathscr{L}_3\\
[\mathscr{L}_3, \mathscr{L}_2] &= \left(\wp_{1;3} + {4 \over 5} \lambda_4 \right) \mathscr{L}_1, \qquad \quad
[\mathscr{L}_3, \mathscr{L}_6] = {3 \over 5} \lambda_8 \mathscr{L}_1 + \wp_{0;3,3} \mathscr{L}_3, \\
[\mathscr{L}_3, \mathscr{L}_4] &= \left(\wp_{0;3,3} + {6 \over 5} \lambda_6 \right) \mathscr{L}_1 + \left(\wp_{1;3} - \lambda_4\right) \mathscr{L}_3, \\
[\mathscr{L}_2, \mathscr{L}_4] &=\frac{8}{5}\lambda_6 \mathscr{L}_0 -\frac{8}{5}\lambda_4 \mathscr{L}_2 + 2 \mathscr{L}_6
-\frac{1}{2}\wp_{2;3} \mathscr{L}_1 +\frac{1}{2}\wp_3 \mathscr{L}_3, & \\
[\mathscr{L}_2, \mathscr{L}_6] &= \frac{4}{5}\lambda_8 \mathscr{L}_0 -\frac{4}{5}\lambda_4 \mathscr{L}_4
-\frac{1}{2}\wp_{1;3,3} \mathscr{L}_1 +\frac{1}{2}\wp_{2;3} \mathscr{L}_3, & \\
[\mathscr{L}_4, \mathscr{L}_6] &=-2\lambda_{10} \mathscr{L}_0 +\frac{6}{5}\lambda_8 \mathscr{L}_2 -\frac{6}{5}\lambda_6 \mathscr{L}_4 + 2\lambda_4 \mathscr{L}_6
- \frac{1}{2}\wp_{0;3,3,3} \mathscr{L}_1 +\frac{1}{2}\wp_{1;3,3} \mathscr{L}_3. 
\end{align*}

\section{Polynomial vector fields in $\mathcal{B}$} \label{s4}

We consider $\mathbb{C}^{2g}$ with coordinates $(\lambda_4, \lambda_6, \ldots, \lambda_{4 g}, \lambda_{4 g + 2})$ and set $\lambda_s = 0$
for every $s \notin \{ 4,6, \ldots, 4 g, 4 g + 2\}$.
For $k, m \in \{ 1, 2, \ldots, 2 g\}$, $k\leqslant m$ set
\[
 T_{2k, 2m} = 2 (k+m) \lambda_{2k+2m} + \sum_{s=2}^{k-1} 2 (k + m - 2 s) \lambda_{2s} \lambda_{2k+2m-2s}
 - {2 k (2 g - m + 1) \over 2 g + 1} \lambda_{2k} \lambda_{2m},
\] 
and for $k > m $ set $T_{2k, 2m} = T_{2m, 2k}$.
For $k = 0, 1, 2, \ldots, {2 g - 1}$ we
have the vector fields 
\begin{equation} \label{Lk}
 L_{2k} = \sum_{s = 2}^{2 g + 1} T_{2k + 2, 2 s - 2} {\partial \over \partial \lambda_{2s}}.
\end{equation}

\eject

\begin{prop} \label{pB} \text{}
\begin{enumerate}
 \item The vector field $L_0$ is the Euler vector field on $\mathcal{B}$.\\
 We have $L_0(\lambda_k) = k \lambda_k$, $[L_0, L_{2k}] = 2k L_{2k}$.
 \item We have $L_{2k} (\lambda_{2s+4}) = L_{2s} (\lambda_{2k+4})$.
 \item The vector fields $L_{2k}$ are independent at any point of $\mathcal{B}$.
 \item The vector fields $L_{2k}$ are tangent to the discriminant curve $\Sigma$ of $\mathcal{V}_\lambda$.
\end{enumerate} 
\end{prop}

Let us give some remarks. Note that properties (1) and (2) follow directly from definition.
The property (3) is the one we will use essentially further in the work. For the property (4) let us introduce the following notation.
Consider the curve $\mathcal{V}_\lambda$. Set 
\[
f(X) = X^{2g+1} + \lambda_4 X^{2 g - 1}  + \lambda_6 X^{2 g - 2} + \ldots + \lambda_{4 g} X + \lambda_{4 g + 2}. 
\]
Denote by $R(\lambda)$ the resultant of $f(X)$ and $f'(X)$. The discriminant curve $\Sigma$ is defined by~$(\lambda \in \Sigma) \Leftrightarrow (R(\lambda) = 0)$.
Thus, property (4) is $L_{2k}(R(\lambda)) = 0$ for $\lambda$ such that $R(\lambda) = 0$.

We will prove the Properties \ref{pB} for the vector fields \eqref{Lk} for $g = 1,2,3$.
Let us note that for general $g$ various constructions have been given to solve the problem of finding vector fields of the form \eqref{Lk} with Properties~\ref{pB} 
(see \cite{BM}, \cite{A} section 4).
We will need the explicit form \eqref{Lk} for these vector fields.

\begin{lem} \label{l1}
The vector fields $L_{2k}$, $k = 0, 1, \ldots, 2 g - 1$, have properties \ref{pB} for $g = 1,2,3$.
\end{lem}

\begin{proof}
Denote by $T$ the symmetric $2 g \times 2 g$ matrix with elements $T_{2k, 2m} = L_{2k -2}(\lambda_{2m+2})$.

For $g = 1$ we have
\begin{equation} \label{T1}
T = \begin{pmatrix}
    4 \lambda_4 & 6 \lambda_6 \\
    6 \lambda_6 & - {4 \over 3} \lambda_4^2
    \end{pmatrix},
    \qquad
\det T = - {4 \over 3} \left(4 \lambda_4^3 + 27 \lambda_6^2\right), \qquad R(\lambda) = 4 \lambda_4^3 + 27 \lambda_6^2.
\end{equation}
We have $(R(\lambda) = 0) \Leftrightarrow (\det T = 0)$, therefore the vector fields $L_0$ and $L_2$ are independent at any point of
$\mathcal{B} = \mathbb{C}^2 \backslash \Sigma$.
We have 
$
L_0 \det T = 12 \det T$, $L_2 \det T = 0,
$
so the vector fields $L_0$ and $L_2$ are tangent to $\Sigma$.

For $g = 2$ we have
\begin{equation} \label{T2}
 T =  \begin{pmatrix}
 4 \lambda_4 & 6 \lambda_6 & 8 \lambda_8 & 10 \lambda_{10}\\
 6 \lambda_6 & 8 \lambda_8 & 10 \lambda_{10} & 0 \\
 8 \lambda_8 & 10 \lambda_{10} & 4 \lambda_4 \lambda_8 & 6 \lambda_4 \lambda_{10} \\
 10 \lambda_{10} & 0 & 6 \lambda_4 \lambda_{10} & 4 \lambda_6 \lambda_{10}
 \end{pmatrix} 
 - {1 \over 5} 
 \begin{pmatrix}
 0 & 0 & 0 & 0 \\
 0 & 12 \lambda_4^2 & 8 \lambda_4 \lambda_6 & 4 \lambda_4 \lambda_8 \\
 0 & 8 \lambda_4 \lambda_6 & 12 \lambda_6^2 & 6 \lambda_6 \lambda_8 \\
 0 & 4 \lambda_4 \lambda_8 & 6 \lambda_6 \lambda_8 & 8 \lambda_8^2 \\
 \end{pmatrix}.
\end{equation}
This coincides with the matrix given in \cite{BL}, example 15, so we give this reference for the proof.
We have as well $(R(\lambda) = 0) \Leftrightarrow (\det T = 0)$.

For $g = 3$ we have
\begin{multline} \label{T3}
 T = \begin{pmatrix}
     4 \lambda_4 & 6 \lambda_6 & 8 \lambda_8 & 10 \lambda_{10} & 12 \lambda_{12} & 14 \lambda_{14} \\
     6 \lambda_6 & 8 \lambda_8 & 10 \lambda_{10} & 12 \lambda_{12} & 14 \lambda_{14} & 0 \\
     8 \lambda_8 & 10 \lambda_{10} & 12 \lambda_{12} + 4 \lambda_4 \lambda_8 & 14 \lambda_{14} + 6 \lambda_4 \lambda_{10} & 8 \lambda_4 \lambda_{12} & 10 \lambda_4 \lambda_{14} \\
     10 \lambda_{10} & 12 \lambda_{12} & 14 \lambda_{14} + 6 \lambda_4 \lambda_{10} & 4 \lambda_6 \lambda_{10} + 8 \lambda_4 \lambda_{12}
     & 6 \lambda_6 \lambda_{12} + 10 \lambda_4 \lambda_{14} & 8 \lambda_6 \lambda_{14} \\
     12 \lambda_{12} & 14 \lambda_{14} & 8 \lambda_4 \lambda_{12} & 6 \lambda_6 \lambda_{12} + 10 \lambda_4 \lambda_{14} & 4 \lambda_8 \lambda_{12} + 8 \lambda_6 \lambda_{14}
     & 6 \lambda_8 \lambda_{14} \\
     14 \lambda_{14} & 0 & 10 \lambda_4 \lambda_{14} & 8 \lambda_6 \lambda_{14} & 6 \lambda_8 \lambda_{14} & 4 \lambda_{10} \lambda_{14} \\
     \end{pmatrix} - \\
     - {1 \over 7}
     \begin{pmatrix}
     0 & 0 & 0 & 0 & 0 & 0 \\
     0 & 20 \lambda_4^2 & 16 \lambda_4 \lambda_6 & 12 \lambda_4 \lambda_8 &  8 \lambda_4 \lambda_{10} & 4 \lambda_4 \lambda_{12} \\
     0 & 16 \lambda_4 \lambda_6 & 24 \lambda_6^2 & 18 \lambda_6 \lambda_8 & 12 \lambda_6 \lambda_{10} & 6 \lambda_6 \lambda_{12} \\
     0 & 12 \lambda_4 \lambda_8 & 18 \lambda_6 \lambda_8 & 24 \lambda_8^2 & 16 \lambda_8 \lambda_{10} & 8 \lambda_8 \lambda_{12} \\
     0 &  8 \lambda_4 \lambda_{10} & 12 \lambda_6 \lambda_{10} & 16 \lambda_8 \lambda_{10} & 20 \lambda_{10}^2 & 10 \lambda_{10} \lambda_{12} \\
     0 &  4 \lambda_4 \lambda_{12} &  6 \lambda_6 \lambda_{12} &  8 \lambda_8 \lambda_{12} & 10 \lambda_{10} \lambda_{12} & 12 \lambda_{12}^2 \\
     \end{pmatrix}.  
\end{multline}

A straightforward calculation gives $\det T = - {64 \over 7} R(\lambda)$ and $(R(\lambda) = 0) \Leftrightarrow (\det T = 0)$,
therefore the vector fields $L_{2k}$ are independent at any point of $\mathcal{B} = \mathbb{C}^6 \backslash \Sigma$.
We have
\[
 (L_0, L_2, L_4, L_6, L_8, L_{10}) \det T = (84, 0, 40 \lambda_4, 24 \lambda_6, 12 \lambda_8, 4 \lambda_{10}) \det T,
\]
so the vector fields $L_{2k}$ are tangent to $\Sigma$.
\end{proof}

\begin{lem} For $g = 3$ we have
\begin{equation} \label{M}
 \begin{pmatrix}
 [L_2, L_4] \\
 [L_2, L_6] \\
 [L_2, L_8] \\
 [L_2, L_{10}] \\
 [L_4, L_6] \\
 [L_4, L_8] \\
 [L_4, L_{10}] \\
 [L_6, L_8] \\
 [L_6, L_{10}] \\
 [L_8, L_{10}] 
 \end{pmatrix}
=
\mathcal{M}
 \begin{pmatrix}
 L_0 \\
 L_2 \\
 L_4 \\
 L_6 \\
 L_8 \\
 L_{10}
 \end{pmatrix}
 \text{for } \mathcal{M} =  {2 \over 7} \begin{pmatrix}
8 \lambda_6 & - 8 \lambda_4 & 0 & 7 & 0 & 0 \\
6 \lambda_8 & 0 & - 6 \lambda_4 & 0 & 14 & 0 \\
4 \lambda_{10} & 0 & 0 & - 4 \lambda_4 & 0 & 21 \\
2 \lambda_{12} & 0 & 0 & 0 & - 2 \lambda_4 & 0 \\
- 7 \lambda_{10} & 9 \lambda_8 & - 9 \lambda_6 & 7 \lambda_4 & 0 & 7 \\
- 14 \lambda_{12} & 6 \lambda_{10} & 0 & - 6 \lambda_6 & 14 \lambda_4 & 0 \\
- 21 \lambda_{14} & 3 \lambda_{12} & 0 & 0 & - 3 \lambda_6 & 21 \lambda_4 \\
- 7 \lambda_{14} & - 7 \lambda_{12} & 8 \lambda_{10} & - 8 \lambda_8 & 7 \lambda_6 & 7 \lambda_4 \\
0 & - 14 \lambda_{14} & 4 \lambda_{12} & 0 & - 4 \lambda_8 & 14 \lambda_6 \\
0 & 0 &- 7 \lambda_{14} & 5 \lambda_{12} & - 5 \lambda_{10} & 7 \lambda_8 \\
 \end{pmatrix}.
\end{equation}
\end{lem}

\begin{proof}
The proof is a straightforward calculation in polynomial vector fields \eqref{Lk}.
\end{proof}

\section{Polynomial map} \label{s5}

Consider the diagram
\begin{equation} 
 \xymatrix{
	\mathcal{U} \ar[d]^{\pi} \ar@{-->}[r]^{\varphi} & \mathbb{C}^{3 g} \ar[d]^{p}\\
	\mathcal{B} \ar@{^{(}->}[r] & \mathbb{C}^{2g}\\
	}   \label{d}
\end{equation}
Here $\pi: \mathcal{U} \to \mathcal{B}$ is the bundle described above, $\mathcal{B} \subset \mathbb{C}^{2g}$ is the embedding given by the coordinates $\lambda$.
The map $\varphi$ will be given by a set of generators of $\mathcal{F}$ and $p$ will be a polynomial map. 
We will now describe $\varphi$ and $p$. 

We use a fundamental result from the theory of hyperelliptic Abelian functions (see~\cite{BEL-12}, Chapter 5):
Any hyperelliptic function can be represented as a rational function
in $\wp_{1;k}$ and $\wp_{2;k}$, where $k \in \{1, 3, \ldots, 2 g - 1\}$.
The following theorem from \cite{BEL} gives a set of relations between the derivatives of these functions.
We use it to determine a set of generators in~$\mathcal{F}$.

\begin{thm}[\cite{BEL}] For $i, k \in \{1, 3, \ldots, 2 g - 1\}$ we have the relations
\begin{align}
\wp_{3; i} = & 6 \wp_2 \wp_{1; i} + 6 \wp_{1; i+2} - 2 \wp_{0;3, i} + 2 \lambda_{4} \delta_{i,1}, \label{2}\\
\wp_{2; i} \wp_{2; k} = & 4 \left(\wp_2 \wp_{1; i} \wp_{1; k} + \wp_{1; k} \wp_{1; i+2} + \wp_{1; i} \wp_{1; k+2} + \wp_{0; k+2, i+2}\right)  
   - \nonumber \\
   & - 2 (\wp_{1; i} \wp_{0;3,k} + \wp_{1; k} \wp_{0;3,i} + \wp_{0; k, i+4} + \wp_{0;i, k+4}) + \label{3} \\
   &+ 2 \lambda_4 (\delta_{i,1} \wp_{1; k} + \delta_{k,1} \wp_{1;i}) + 2 \lambda_{i+k+4} (2 \delta_{i,k} + \delta_{k, i-2} + \delta_{i, k-2}). \nonumber
\end{align}
\end{thm}
\begin{proof}
 In \cite{BEL} we have formulas (4.1) and (4.8). Using the notation \eqref{n} we get \eqref{2} from~(4.1) and \eqref{3} from (4.8).
\end{proof}

\eject

\begin{cor} \label{c1}
 Consider the map $\varphi: \mathcal{U} \dashrightarrow \mathbb{C}^{ \frac{g (g+9)}{2}}$, where in $\mathbb{C}^{ \frac{g (g+9)}{2}}$ we denote
 the coordinates by
 $(x, w, \lambda)$. Here for $x = (x_{i,j}) \in \mathbb{C}^{3g}$ we have $i \in \{ 1,2,3 \}$, \mbox{$j \in \{1, 3, \ldots, 2 g -1\}$,}
 for $w = (w_{k,l}) \in \mathbb{C}^{\frac{g (g-1)}{2}}$ we have $k,l \in \{3, 5, \ldots, 2 g -1\}$, $k \leqslant l$,
 and for $\lambda = (\lambda_s) \in \mathbb{C}^{2 g}$ we have $s \in \{4, 6, \ldots, 4 g, 4 g + 2\}$.
 For $(u, \lambda) = (u_1, u_3, \ldots, u_{2g-1}, \lambda_4, \lambda_6, \ldots, \lambda_{4 g + 2})$
 set 
 \[
  \varphi: (u, \lambda) \mapsto (x_{i,j}, w_{k,l}, \lambda_s) = (\wp_{i;j}(u, \lambda), \wp_{0;k,l}(u, \lambda), \lambda_s).
 \]
 We also denote $x_{i+1} = x_{i,1}$, $i = 1,2,3$, and $w_{l,k} = w_{k,l}$.
 
Then the image of $\varphi$ lies in $\mathcal{S} \subset \mathbb{C}^{ \frac{g (g+9)}{2}}$, where $\mathcal{S}$ is determined by the set of $\frac{g (g+3)}{2}$ equations
\begin{align}
x_{4} = & 6 x_{2}^2 + 4 x_{1, 3} + 2 \lambda_{4}, \label{e1} \\
x_{3, k} = & 6 x_{2} x_{1, k} + 6 x_{1, k+2} - 2 w_{3, k}, \label{e2} \\
x_{3}^2 = & 4 x_2^3 + 4 x_{2} x_{1, 3} - 4 x_{1, 5} + 4 w_{3, 3} + 4 \lambda_4 x_{2} + 4 \lambda_{6}, \label{e3} \\
x_{3} x_{2, k} = & 4 x_2^2 x_{1, k} + 2 x_{1, 3} x_{1, k} + 4 x_{2} x_{1, k+2} - 2 x_{1, k+4} + \label{4} \\
   & - 2 x_{2} w_{3,k} + 4 w_{3, k+2} - 2 w_{5, k} + 2 \lambda_4 x_{1, k} + 2 \lambda_{8} \delta_{3, k}, \nonumber \\
x_{2,j} x_{2,k} = & 4 x_2 x_{1, j} x_{1, k} + 4 x_{1, k} x_{1, j+2} + 4 x_{1, j} x_{1, k+2} + 4 w_{k+2, j+2}
   - \label{5} \\
   & - 2 x_{1, j} w_{3,k} - 2 x_{1, k} w_{3,j} - 2 w_{k, j+4} - 2 w_{j, k+4} + 2 \lambda_{j+k+4} (2 \delta_{j,k} + \delta_{k, j-2} + \delta_{j, k-2}), \nonumber
\end{align}
for $j,k \in \{3, \ldots, 2g-1\}$ and any variable equal to zero if the index is out of range.
\end{cor}

\begin{thm} \label{thm3}
The projection $\pi_1\colon \mathbb{C}^{\frac{g (g+9)}{2}} \to \mathbb{C}^{3 g}$ on the first $3 g$ coordinates gives the isomorphism $\mathcal{S} \simeq \mathbb{C}^{3g}$.
Therefore, the coordinates $x$ uniformize $\mathcal{S}$.
\end{thm}

\begin{cor} \label{cor2}
The projection $\pi_2\colon \mathbb{C}^{\frac{g (g+9)}{2}} \to \mathbb{C}^{\frac{g (g-1)}{2}}$ on the second $\frac{g (g-1)}{2}$ coordinates
gives a polynomial map $\mathbb{C}^{3g} \to \mathbb{C}^{\frac{g (g+9)}{2}}$.
\end{cor}

\begin{cor} \label{cor3}
The projection $\pi_3\colon \mathbb{C}^{\frac{g (g+9)}{2}} \to \mathbb{C}^{2 g}$ on the last $2 g$ coordinates gives a polynomial map $p\colon \mathbb{C}^{3g} \to \mathbb{C}^{2g}$.
\end{cor}

We obtain the diagram:
\[
\xymatrix{
	 & \mathbb{C}^{\frac{g (g+9)}{2}} \ar@{<-_{)}}[d] \ar@{=}[r] & \mathbb{C}^{3 g} \times \mathbb{C}^{\frac{g (g-1)}{2}} \times \mathbb{C}^{2 g} \ar[dl]^{\pi_1} \ar@/^/[ddl]^{\pi_3}\\
	\mathcal{U} \ar[d]^{\pi} \ar@{-->}[r]^(.33){\varphi}& \mathcal{S} \simeq \mathbb{C}^{3 g} \ar[d]^{p}\\
	\mathcal{B} \ar@{^{(}->}[r] & \mathbb{C}^{2g}\\
	} 
\]

\begin{proof}[Proof of theorem \ref{thm3}]

Equations \eqref{e1}--\eqref{e3} give
\begin{equation} 
\lambda_{4} = \frac{1}{2} x_{4} - 3 x_{2}^2 - 2 x_{1, 3}, \quad \lambda_{6} = \frac{1}{2} x_{3, 3} - 2 x_{1, 5} + \frac{1}{4} x_{3}^2 + 2 x_2^3 - 2 x_{2} x_{1, 3} - \frac{1}{2} x_{2} x_{4}, \label{6} \\
\end{equation}
\begin{equation} 
w_{3, k} = - \frac{1}{2} x_{3, k} + 3 x_{2} x_{1, k} + 3 x_{1, k+2}. \qquad \qquad \label{y3}
\end{equation}
For $k = 3$ equation \eqref{4} with \eqref{6} gives
\[
\lambda_{8} = - \frac{1}{2} (x_{2} x_{3, 3} - x_{3} x_{2, 3} + x_{4} x_{1, 3}) + (4 x_{2}^2 + x_{1, 3}) x_{1, 3} + \frac{1}{2} x_{3, 5} - 2 (x_{2} x_{1, 5} + x_{1, 7}).
\]
For $k \geqslant 5$ equation \eqref{4} with \eqref{6} gives
\begin{equation} \label{y5}
w_{5, k} = \frac{1}{2} (x_{2} x_{3, k} - x_{3} x_{2, k} + x_{4} x_{1, k}) - \left(4 x_{2}^2 + x_{1, 3}\right) x_{1, k} - x_{3, k+2} + 5 (x_{2} x_{1, k+2} + x_{1, k+4}).
\end{equation}

Now for equation \eqref{5} we can assume $j \geqslant k$. For $j = k$, $j = k + 2$ and $j \geqslant k+2$ respectively we obtain the equations
\begin{align}
 x_{2,k}^2 = & 4 x_2 x_{1, k}^2 + 8 x_{1, k} x_{1, k+2} - 4 x_{1, k} w_{3,k} + 4 w_{k+2, k+2} - 4 w_{k, k+4} + 4 \lambda_{2k+4}, \label{l4} \\
   x_{2,k+2} x_{2,k} = & 4 \left(x_2 x_{1, k+2} x_{1, k} + x_{1, k} x_{1, k+4} + x_{1, k+2} x_{1, k+2}\right) 
   - \label{l6} \\
   & - 2 (x_{1, k+2} w_{3,k} + x_{1, k} w_{3,k+2})  + 4  w_{k+2, k+4} - 2 w_{k, k+6} - 2 w_{k+2, k+4} + 2 \lambda_{2 k+ 6}, \nonumber \\
    x_{2,j} x_{2,k} = & 4 \left(x_2 x_{1, j} x_{1, k} + x_{1, k} x_{1, j+2} + x_{1, j} x_{1, k+2}\right) - 2 (x_{1, j} w_{3,k} + x_{1, k} w_{3,j})
    + \label{yk} \\
   & + 4 w_{k+2, j+2} - 2 w_{k, j+4} - 2 w_{j, k+4}. \nonumber
\end{align}
Now \eqref{y3}, \eqref{y5} and \eqref{yk} for $k = 3, 5, 7, \ldots$ give expressions for $w_{j, k+4}$ in terms of $x$ while 
\eqref{l4} and \eqref{l6} give expressions for $\lambda_{2k+4}$ and $\lambda_{2 k+ 6}$ in terms of $x$ and $w$.
We see that the expressions for $w$ and $\lambda$ are polynomial in $x$, which proves also Corollaries \ref{cor2} and~\ref{cor3}.
\end{proof}

\section{Generators of the polynomial Lia algebra in $\mathbb{C}^{3 g}$} \label{s6}

Now let us return to Problem \ref{p1}. The $\mathcal{F}$-module $\Der \mathcal{F}$ is determined by it's action on the generators of $\mathcal{F}$.
We take the $3 g$ functions $\wp_{i;j}(u, \lambda)$, where $i \in \{ 1,2,3 \}$ \mbox{and~$j \in \{1, 3, \ldots, 2 g -1\}$,} as these generators.
The values of these functions determine~$\lambda$ by Theorem~\ref{thm3}.

Denote the ring of polynomials in $\lambda \in \mathbb{C}^{2g}$ by $\mathcal{P}$.
Let us consider the polynomial map~$p\colon \mathbb{C}^{3g} \to \mathbb{C}^{2g}$.
A vector field $\mathcal{L}$ in $\mathbb{C}^{3g}$ will be called projectable for $p$ if there exists a vector field $L$ in $\mathbb{C}^{2g}$ such that
\[
 \mathcal{L}(p^* f) = p^* L(f) \quad \text{for any} \quad f \in \mathcal{P}.
\]
The vector field $L$ will be called the pushforward of $\mathcal{L}$.
A corollary of this definition is that for a projectable vector field $\mathcal{L}$ we have $\mathcal{L}(p^* \mathcal{P}) \subset p^* \mathcal{P}$.

Now we will consider the following problem.

\begin{prb} \label{p6}
Find $3 g$ polynomial vector fields in $\mathbb{C}^{3g}$ projectable for $p\colon \mathbb{C}^{3g} \to \mathbb{C}^{2g}$ and
independent at any point in $p^{-1}(\mathcal{B})$.
Construct their polynomial Lie algebra.
\end{prb}

Let us explain the connection to Problem \ref{p1}.
Given a solution to Problem \ref{p6} for each of the $3g$ vector fields~$\mathcal{L}_k$
with pushforwards $L_k$ we will restore the vector fields~$\mathscr{L}_k$ 
projectable for $\pi$ with pushforwards $L_k$ and such that $\mathscr{L}_k (\varphi^* x_{i,j}) = \varphi^* \mathcal{L}_k (x_{i,j})$
for the coordinate functions $x_{i,j}$ in $\mathbb{C}^{3g}$.
As $\varphi^* x_{i,j}$ are the generators of $\mathcal{F}$ and $\mathcal{L}_k (x_{i,j})$ is a polynomial in $x_{i,j}$,
this gives $\mathscr{L}_k (\varphi^* x_{i,j}) \in \mathcal{F}$ and $\mathscr{L}_k \in \Der \mathcal{F}$.

Now we will give a solution to Problem \ref{p6} for $g = 1,2,3$.

We denote by $\mathcal{T}$ the $3 g \times 3 g$ matrix with elements $\mathcal{L}_k(x_{i,j})$ for $k = \{ 0, 2, \ldots, 4 g - 2\} \cup \{ 1, 3, \ldots 2 g - 1\}$ and 
$i \in \{ 1,2,3 \}$, $j \in \{1, 3, \ldots, 2 g -1\}$.

The vector field $\mathcal{L}_0$ is the Euler vector field on $\mathbb{C}^{3g}$, we have
\begin{align}
\mathcal{L}_0 &= \sum_j (j+1) x_{1,j} {\partial \over \partial x_{1,j}} + (j+2) x_{2,j} {\partial \over \partial x_{2,j}} +
(j+3) x_{3, j} {\partial \over \partial x_{3,j}}. \label{L0}
\end{align}

Then
$\mathcal{L}_0(p^* f) = p^* L_0(f)$ for any $f \in \mathcal{P}$.

Denote by $\mathcal{L}_k$ for $k \in \{1, 3, \ldots, 2 g -1\}$ polynomial vector fields in $\mathbb{C}^{3g}$
such that $\mathscr{L}_k(\varphi^*(x_{i,j})) = \varphi^* \mathcal{L}_k(x_{i,j})$. We have $\mathcal{L}_k(p^* f) = 0$ for any $f \in \mathcal{P}$.

\begin{lem}
 We have 
\begin{align}
\mathcal{L}_1 &= \sum_j x_{2,j} {\partial \over \partial x_{1,j}} + x_{3, j} {\partial \over \partial x_{2,j}} + 4 (2 x_{2} x_{2,j} + x_{3} x_{1, j} + x_{2, j+2}) {\partial \over \partial x_{3,j}}
\label{L1}
\end{align}
where $x_{2,2 g + 1} = 0$. 
\end{lem}

\begin{proof}
By definition we have $\mathcal{L}_1(x_{1,k}) = x_{2,k}$ and $\mathcal{L}_1(x_{2,k}) = x_{3,k}$.
Denote the polynomial $\mathcal{L}_1(x_{3,k})$ by $x_{4,k}$. We have $\varphi^* x_{4,k} = \wp_{4;k}$.
By definition we have \mbox{$\mathcal{L}_k(x_{2}) = x_{2,k}$,} \mbox{$\mathcal{L}_k(x_{3}) = x_{3,k}$,} $\mathcal{L}_k(x_{4}) = x_{4,k}$.
Denote the polynomial $\mathcal{L}_k(x_{1,3})$ by $p_{1,3,k}$. Therefore, we have \mbox{$\varphi^* p_{1,3,k} = \wp_{1;3,k}$} and $\mathcal{L}_1(w_{3,k}) = p_{1,3,k}$.

Applying $\mathcal{L}_k$  and $\mathcal{L}_1$ to \eqref{e1} and \eqref{e2} we get
\begin{align*}
x_{4,k} = & 12 x_{2} x_{2,k} + 4 p_{1, 3, k}, &
x_{4,k} = & 6 x_{3} x_{1, k} +  6 x_{2} x_{2, k} + 6 x_{2, k+2} - 2 p_{1,3, k}, 
\end{align*}
thus $x_{4,k} = 8 x_{2} x_{2,k} + 4 x_{3} x_{1, k} + 4 x_{2, k+2}$.
\end{proof}

\begin{lem}
 For $s = 3, 5, \ldots, 2g - 1$ we have 
\begin{align}
\mathcal{L}_{s} &= x_{2,s} {\partial \over \partial x_{2}} + x_{3,s} {\partial \over \partial x_{3}} \label{Ls}
+ \mathcal{L}_1(x_{3,s}) {\partial \over \partial x_{4}} + \\ & +
\sum_{k=1}^{g-1} \mathcal{L}_1(w_{s,2k+1}) {\partial \over \partial x_{1,2k+1}} + \mathcal{L}_1(\mathcal{L}_1(w_{s,2k+1})) {\partial \over \partial x_{2,2k+1}}
+ \mathcal{L}_1(\mathcal{L}_1(\mathcal{L}_1(w_{s,2k+1}))) {\partial \over \partial x_{3,2k+1}}. \nonumber
\end{align}
\end{lem}
\begin{proof}
By definition the vector fields $\mathcal{L}_1$ and $\mathcal{L}_{s}$ for $s = 3, 5, \ldots, 2g - 1$ commute. We have
$\mathcal{L}_{s}(x_i) = x_{i,s}$ and $\mathcal{L}_{s}(x_{1,2k+1}) = \mathcal{L}_{1}(w_{s,2k+1})$. This determines the coefficients.
\end{proof}

\begin{cor} \label{cLs}
Knowing $\mathcal{L}_1$ and $\mathcal{L}_{s}(x_{1,j})$ determines the coefficients  $\mathcal{L}_{s}(x_{2,j})$ and  $\mathcal{L}_{s}(x_{3,j})$.
\end{cor}

\section{Genus 1} \label{s7}

In this section we give polynomial vector fields satisfying Problem \ref{p6} in the case $g = 1$.

The map $p$ takes the form
\begin{align*}
\lambda_4 &= - 3 x_2^2 + {1 \over 2} x_4, &
\lambda_6 &= 2 x_2^3 + {1 \over 4} x_3^2 - {1 \over 2}  x_2 x_4.
\end{align*}
The vector fields are
\[
\begin{pmatrix}
\mathcal{L}_0\\
\mathcal{L}_1\\
\mathcal{L}_2
\end{pmatrix}
=
\begin{pmatrix}
2 x_2 & 3 x_3 & 4 x_4 \\
x_3 & x_4 & 12 x_2 x_3\\
{2\over3} x_4 - 2 x_2^2 & 3 x_2 x_3 & 2 x_2 x_4 + 3 x_3^2
\end{pmatrix}
\begin{pmatrix}
{\partial \over \partial x_2}\rule[-7pt]{0pt}{16pt}\\
{\partial \over \partial x_3}\rule[-7pt]{0pt}{16pt}\\
{\partial \over \partial x_4}\rule[-7pt]{0pt}{16pt}
\end{pmatrix}.
\]
The middle matrix is $\mathcal{T}$. We have
$\det \mathcal{T} = 4 \det T$ for $T$ determined by \eqref{T1}. Therefore, the vector fields $\mathcal{L}_k$ are independent at any point of $p^{-1}(\mathcal{B})$.

The polynomial Lie algebra is
\begin{align*}
[\mathcal{L}_0, \mathcal{L}_1] &= \mathcal{L}_1, & [\mathcal{L}_0, \mathcal{L}_2] &= 2 \mathcal{L}_2, &
[\mathcal{L}_1, \mathcal{L}_2] &= x_2 \mathcal{L}_1.
\end{align*}

\section{Genus 2} \label{s8}
In this section we give polynomial vector fields satisfying Problem \ref{p6} in the case $g = 2$.

In the formulas below we write $y_4$ instead of $x_{1,3}$, $y_5$ instead of $x_{2,3}$ and $y_6$ instead of~$x_{3,3}$ to shorten down the formulas.

The map $p$ takes the form
\begin{align}
\lambda_{4} &= - 3 x_{2}^2 + {1 \over 2} x_{4} - 2 y_4, & \lambda_{6} &= 2 x_{2}^3 + {1 \over 4} x_{3}^2 - {1 \over 2}  x_{2} x_{4} - 2 x_{2} y_4 + {1 \over 2} y_6,
\nonumber \\
\lambda_{8} &= \left(4 x_{2}^2 + y_4\right) y_4 - {1 \over 2} (x_{4} y_4 - x_{3} y_5 + x_{2} y_6), &
\quad \lambda_{10} &= 2 x_{2} y_4^2 + {1 \over 4} y_5^2 - {1 \over 2} y_4 y_6. \label{l2}
\end{align}
By \eqref{L0} and \eqref{L1} we have
\begin{align*}
\mathcal{L}_0 &= 2 x_{2} {\partial \over \partial x_{2}} + 3 x_{3} {\partial \over \partial x_{3}} +
4 x_{4} {\partial \over \partial x_{4}} + 4 y_4 {\partial \over \partial y_4} + 5 y_5 {\partial \over \partial y_5} +
6 y_6 {\partial \over \partial y_6},\\
\mathcal{L}_1 &= 
x_3 {\partial \over \partial x_2} + x_4 {\partial \over \partial x_3}
+ 4 (3 x_2 x_3 + y_5) {\partial \over \partial x_4}
+ y_5 {\partial \over \partial y_4} + y_6 {\partial \over \partial y_5}
+ 4 (2 x_2 y_5 + x_3 y_4) {\partial \over \partial y_6}.
\end{align*}
According to \eqref{Ls} and Corollary \ref{cLs} the vector field $\mathcal{L}_3$ is determined by
\begin{align*}
\mathcal{L}_3(x_2) &= y_5, & \mathcal{L}_3(y_4) &= x_3 y_4 - x_2 y_5.
\end{align*}
We have
\begin{align*}
\mathcal{L}_3(x_3) &= y_6, & \mathcal{L}_3(y_5) &= x_4 y_4 - x_2 y_6,\\
\mathcal{L}_3(x_4) &= 4 (2 x_2 y_5 + x_3 y_4), & \mathcal{L}_3(y_6) &= 8 x_2 x_3 y_4 - 8 x_2^2 y_5 + x_4 y_5 - x_3 y_6 + 4 y_4 y_5.
\end{align*}

Set $p_7 = \mathcal{L}_1(w_6) = \mathcal{L}_3(y_4)$, $w_9 = \mathcal{L}_3(w_6)$.
By \eqref{y3} we have $w_6 = 3 x_2 y_4 - {1 \over 2} y_6$,
\begin{align*}
p_7 &= x_3 y_4 - x_2 y_5, &
w_9 &= - x_2 x_3 y_4  + x_2^2 y_5 - {1 \over 2} \left(x_4 y_5 - x_3 y_6\right) + y_4 y_5.
\end{align*}

Now we introduce the vector fields
$\widehat{\mathcal{L}}_2, \widehat{\mathcal{L}}_4, \widehat{\mathcal{L}}_6$.
They are determined up to constants \mbox{$\alpha$, $\beta$, $\gamma_1$} and $\gamma_2$ by $\widehat{\mathcal{L}}_2 = \mathcal{L}_2$,
$\widehat{\mathcal{L}}_4 = \mathcal{L}_4 + \alpha x_3 \mathcal{L}_1$ and
$\widehat{\mathcal{L}}_6 = \mathcal{L}_6 + \beta x_3 \mathcal{L}_3 + (\gamma_1 y_5 + \gamma_2 x_2 x_3) \mathcal{L}_1$,
where
\begin{align*}
\mathcal{L}_2(x_2) &= {8 \over 5} \lambda_{4} + 2 x_2^2 + 4 y_4, \qquad \quad \, \mathcal{L}_2(y_4) = - {4 \over 5} \lambda_{4} x_2 + 2 x_2 y_4,\\
\mathcal{L}_2(x_3) &= 3 x_2 x_3 + 5 y_5, \qquad \qquad \quad \mathcal{L}_2(y_5) = - {4 \over 5} \lambda_{4} x_3 + 3 x_3 y_4,\\
\mathcal{L}_2(x_4) &= 2 x_2 x_4 + 3 x_3^2 + 6 y_6, \qquad \;
\mathcal{L}_2(y_6) = - {4 \over 5} \lambda_{4} x_4 + 4 x_4 y_4 + 3 x_3 y_5 - 2 x_2 y_6,\\
\mathcal{L}_4(x_2) &= {2 \over 5} \lambda_{6} - 2 x_2 y_4 + y_6,\quad 
\mathcal{L}_4(x_3) = x_3 y_4 + 5 x_2 y_5, \quad 
\mathcal{L}_4(x_4) = 6 x_3 y_5 + 4 x_2 y_6,\\
\mathcal{L}_4(y_4) &= - {6 \over 5} \lambda_{6} x_2 + 2 \lambda_{4} y_4 - 4 x_2^2 y_4 + x_4 y_4 - {1 \over 2} x_3 y_5,\\
\mathcal{L}_4(y_5) &= - {6 \over 5} \lambda_{6} x_3 + 2 \lambda_{4} y_5 + 2 x_2 x_3 y_4 - 2 x_2^2 y_5 + 4 y_4 y_5 - w_9,\\
\mathcal{L}_4(y_6) &= - {6 \over 5} \lambda_{6} x_4 + 2 \lambda_{4} y_6 + x_3^2 y_4 + 2 x_2 x_4 y_4 - x_2 x_3 y_5 - 2 x_2^2 y_6 + 5 y_5^2 + 2 y_4 y_6,\\
\mathcal{L}_6(x_2) &= {1 \over 5} \lambda_{8} + {1 \over 2} \left(x_4 y_4 - x_2 y_6\right) - y_4^2, \qquad \qquad
\mathcal{L}_6(x_3) = 3 x_2 p_7 - w_9,\\
\mathcal{L}_6(x_4) &= 2 x_3^2 y_4 + 4 x_2 x_4 y_4 - 2 x_2 x_3 y_5 - 4 x_2^2 y_6 + y_5^2 - 2 y_4 y_6,\\
\mathcal{L}_6(y_4) &= - {8 \over 5} \lambda_{8} x_2 + 2 \lambda_{6} y_4 - 2 x_2 y_4^2 - y_5^2 + y_4 y_6,\\
\mathcal{L}_6(y_5) &= - {8 \over 5} \lambda_{8} x_3 + 2 \lambda_{6} y_5 + x_3 y_4^2 + 5 x_2 y_4 y_5 - y_5 y_6,\\
\mathcal{L}_6(y_6) &= - {8 \over 5} \lambda_{8} x_4 + 2 \lambda_{6} y_6 + 3 x_3 y_4 y_5 - 3 x_2 y_5^2 + 6 x_2 y_4 y_6 - y_6^2.
\end{align*}

\begin{lem}
The vector fields $\mathcal{L}_k$, $k = 0,1,3$, and $\widehat{\mathcal{L}}_k$, $k = 2,4,6$, give a solution to Problem \ref{p6}.
\end{lem}

\begin{proof}
We have $\mathcal{L}_1(p^*(f)) = \mathcal{L}_3(p^*(f)) = 0$ for any $f \in \mathcal{P}$. So for the proof it is sufficient to check
the condition 
$
\widehat{\mathcal{L}}_k(p^*(f)) = p^*(L_k(f))
$ for the generators $f = \lambda_4, \lambda_6, \lambda_8, \lambda_{10}$ of $\mathcal{P}$
in the case $\alpha = \beta = \gamma_1 = \gamma_2 = 0$. This is a straightforward calculation using the explicit polynomial vector fields
$\mathcal{L}_2$, $\mathcal{L}_4$, $\mathcal{L}_6$ and $L_2$, $L_4$, $L_6$ (see Section \ref{s4})
and their action on the polynomials~\eqref{l2}.

We have
$\det \mathcal{T} = - 16 \det T$ for $T$ determined by \eqref{T2}.
Therefore, the vector fields $\widehat{\mathcal{L}}_k$ are independent at any point of $p^{-1}(\mathcal{B})$.
\end{proof}

The polynomial Lie algebra is
\begin{align*}
[\mathcal{L}_0, \mathcal{L}_k] &= k \mathcal{L}_k, \qquad \qquad \qquad [\mathcal{L}_0, \widehat{\mathcal{L}}_k] = k \widehat{\mathcal{L}}_k,
\qquad \qquad \qquad [\mathcal{L}_1, \mathcal{L}_3] = 0, \\
[\mathcal{L}_1, \widehat{\mathcal{L}}_2] &= x_2 \mathcal{L}_1 - \mathcal{L}_3, \qquad \quad \;
[\mathcal{L}_1, \widehat{\mathcal{L}}_4] = y_4 \mathcal{L}_1 + x_2 \mathcal{L}_3 + \alpha x_4 \mathcal{L}_1, \\
[\mathcal{L}_1, \widehat{\mathcal{L}}_6] &= y_4 \mathcal{L}_3 + \left(\gamma_2 (x_3^2 + x_2 x_4) + \gamma_1 y_6\right) \mathcal{L}_1 +
\beta x_4 \mathcal{L}_3,\\
\quad [\mathcal{L}_3, \widehat{\mathcal{L}}_2] &= \left(y_4 + {4 \over 5} \lambda_{4} \right) \mathcal{L}_1, \quad \,
[\mathcal{L}_3, \widehat{\mathcal{L}}_4] = \left(w_6 + {6 \over 5} \lambda_{6}\right) \mathcal{L}_1
+ (y_4 - \lambda_{4}) \mathcal{L}_3 + \alpha y_6 \mathcal{L}_1,\\
[\mathcal{L}_3, \widehat{\mathcal{L}}_6] &=
{3 \over 5} \lambda_8 \mathcal{L}_1 + w_6 \mathcal{L}_3 + \left(\gamma_1 x_4 y_4 + \gamma_2 x_3 y_5 - (\gamma_1 - \gamma_2) x_2 y_6 \right) \mathcal{L}_1 + \beta  y_6 \mathcal{L}_3 ,\\
[\widehat{\mathcal{L}}_2, \widehat{\mathcal{L}}_4] &= {8 \over 5} \lambda_{6} \mathcal{L}_0 - {8 \over 5} \lambda_{4} \widehat{\mathcal{L}}_2
  + 2 \widehat{\mathcal{L}}_6 - {1 \over 2} y_5 \mathcal{L}_1 + {1 \over 2} x_3 \mathcal{L}_3 + \\
& \quad + \left(2 \left(\alpha - \gamma_2\right) x_2 x_3 + \left(5 \alpha - 2 \gamma_1\right) y_5 \right) \mathcal{L}_1 + \left(\alpha - 2 \beta\right) x_3 \mathcal{L}_3,
\\
[\widehat{\mathcal{L}}_2, \widehat{\mathcal{L}}_6] &= {4 \over 5} \lambda_{8} \mathcal{L}_0 - {4 \over 5} \lambda_{4} \widehat{\mathcal{L}}_4 - {1 \over 2} p_7 \mathcal{L}_1 + {1 \over 2} y_5  \mathcal{L}_3 +
 \\ & \quad + 
{1 \over 5} \left( 2 \left(\alpha - \beta  - \gamma_1 \right) \left(x_4 - 6 x_2^2\right)
+ 4 \gamma_2 (x_4 - x_2^2 + y_4) - \left(8 \alpha - 3 \beta - 23 \gamma_1 \right) y_4 \right) x_3 \mathcal{L}_1 - \\
& \quad - \left(\gamma_1 - 5 \gamma_2\right) x_2 y_5 \mathcal{L}_1 + \left(\left(5 \beta + \gamma_1 \right) y_5 + \left(3 \beta + \gamma_2\right) x_2 x_3\right) \mathcal{L}_3,\\
[\widehat{\mathcal{L}}_4, \widehat{\mathcal{L}}_6] &= - 2 \lambda_{10} \mathcal{L}_0 + {6 \over 5} \lambda_8 \widehat{\mathcal{L}}_2
- {6 \over 5} \lambda_6 \widehat{\mathcal{L}}_4 + 2 \lambda_4 \widehat{\mathcal{L}}_6 - {1 \over 2} w_9 \mathcal{L}_1 + {1 \over 2} p_7 \mathcal{L}_3 + \\
& \quad + (\alpha - \beta - \gamma_1 + 2 \gamma_2) \left({6 \over 5} \lambda_{6} + w_6 + y_6\right) x_3 \mathcal{L}_1 + \\
& \quad + \left( 2 \gamma_2 x_2^3 - {1 \over 2} \gamma_2 x_3^2 + (6 \gamma_1 - 7 \alpha) x_2 y_4 + (\beta - \gamma_2) y_6 \right) x_3 \mathcal{L}_1 + \\
& \quad + \left( (4 \alpha - 3 \gamma_1 + 5 \gamma_2) x_2^2 - {1 \over 2} (\alpha - \gamma_1) x_4 + (\alpha + 2 \gamma_1) y_4 \right) y_5 \mathcal{L}_1 + \\
& \quad + \alpha \left( \gamma_2 x_3^3 - \gamma_1 x_4 y_5 - (\beta - \gamma_1) x_3 y_6\right) \mathcal{L}_1 + \\
& \quad +
\left((3 \beta - \gamma_2) x_2^2 x_3 + {1 \over 2} \beta (2 \alpha - 1) x_3 x_4 + (\alpha + 2 \beta) x_3 y_4 + (5 \beta - \gamma_1) x_2 y_5
\right) \mathcal{L}_3. 
\end{align*}
\vspace{-5mm}

\begin{rem}
We give here the full form of the polynomial Lie algebra for arbitrary constants $\alpha, \beta, \gamma_1, \gamma_2$, 
as for any of these constants the polynomial vector fields satisfy Problem \ref{p6} in the case $g = 2$ and thus give a solution to Problem \ref{p1}.
To specify these constants some complimentary condition is needed.
In \cite{BM} the condition that the vector fields form a Witt algebra is taken.
In the current work 
to specify one Lie algebra in this space of parameters we use the following condition on the polynomial Lie algebra: \vspace{-1mm}
\[
 \begin{pmatrix}
 [\mathcal{L}_1, \mathcal{L}_0] \\
 [\mathcal{L}_1, \mathcal{L}_2] \\
 [\mathcal{L}_1, \mathcal{L}_4] \\
 [\mathcal{L}_1, \mathcal{L}_6] \\
 \end{pmatrix}
= 
 \begin{pmatrix}
 -1 & 0 \\
 x_2 & -1 \\
 y_4 & x_2 \\
 0 & y_4 \\
 \end{pmatrix}
 \begin{pmatrix}
  \mathcal{L}_1 \\ \mathcal{L}_3
 \end{pmatrix}.
\]
\vspace{-2mm}

Let us note that this condition leads to the polynomial Lie algebra obtained in \cite{B} and the solution to Problem \ref{p1} described in Section \ref{ss2}.
\end{rem}

\section{Genus 3} \label{s9}

In this section we give polynomial vector fields satisfying Problem \ref{p6} in the case $g = 3$.

Analogously to the genus $g = 2$ case, the vector fields are determined ambiguously up to some parameters. 
This parameters are determined using the condition \eqref{com1}.

In the formulas below we write $y_4$ instead of $x_{1,3}$, $y_5$ instead of~$x_{2,3}$, $y_6$ instead of~$x_{3,3}$,
$z_6$ instead of~$x_{1,5}$, $z_7$ instead of~$x_{2,5}$, $z_8$ instead of $x_{3,5}$,
$w_6$ instead of~$w_{3,3}$, $w_8$ instead of~$x_{3,5}$, and $w_{10}$ instead of~$w_{5,5}$
to shorten down the formulas.

The map $p$ takes the form
\begin{align*}
\lambda_{4} &= - 3 x_2^2 + {1 \over 2} x_4 - 2 y_4, \qquad
\lambda_{6} = 2 x_2^3 + {1 \over 4} x_3^2 - {1 \over 2} x_2 x_4 - 2 x_2 y_4 + {1 \over 2} y_6 - 2 z_6, \\
\lambda_{8} &= 4 x_2^2 y_4 - {1 \over 2} (x_4 y_4 - x_3 y_5 + x_2 y_6) + y_4^2 - 2 x_{2} z_6 + {1 \over 2} z_8, \\
\lambda_{10} &= 2 x_2 y_4^2 + {1 \over 4} y_5^2 - {1 \over 2} y_4 y_6 - \frac{1}{2} (x_4 z_6 - x_3 z_7 + x_2 z_8) + \left(4 x_2^2 + 2 y_4\right) z_6, \\
\lambda_{12} &= 4 x_2 y_4 z_6 - {1 \over 2} (y_6 z_6 - y_5 z_7 + y_4 z_8) + z_6^2, \qquad
\lambda_{14} = 2 x_2 z_6^2 + {1 \over 4} z_7^2 - {1 \over 2} z_6 z_8.
\end{align*}

By \eqref{L0} we have
\begin{align*}
\mathcal{L}_0 &= \sum_{k = 2,3,4} k x_{k} {\partial \over \partial x_{k}} + \sum_{k = 4,5,6} k y_{k} {\partial \over \partial y_{k}} +
\sum_{k = 6,7,8} k z_{k} {\partial \over \partial z_{k}}.
\end{align*}
It is Euler vector field, and for all $k$ we will have 
\[
 [\mathcal{L}_0, \mathcal{L}_k] = k \mathcal{L}_k.
\]
By \eqref{L1} we have
\begin{multline*}
\mathcal{L}_1 = 
x_3 {\partial \over \partial x_2} + x_4 {\partial \over \partial x_3}
+ 4 (3 x_2 x_3 + y_5) {\partial \over \partial x_4} + \\
+ y_5 {\partial \over \partial y_4} + y_6 {\partial \over \partial y_5}
+ 4 (x_3 y_4 + 2 x_2 y_5 + z_7) {\partial \over \partial y_6} + \\
 + z_7 {\partial \over \partial z_6} + z_8 {\partial \over \partial z_7}
+ 4 (x_3 z_6 + 2 x_2 z_7) {\partial \over \partial z_8}.
\end{multline*}
According to definition and Corollary \ref{cLs}, we have
 \[
[\mathcal{L}_1, \mathcal{L}_3] = 0, \quad [\mathcal{L}_1, \mathcal{L}_5] = 0, \quad [\mathcal{L}_3, \mathcal{L}_5] = 0,
 \]
and
the vector fields $\mathcal{L}_3$ and $\mathcal{L}_5$ are determined by
\begin{align*}
\mathcal{L}_3(x_2) &= y_5, & \mathcal{L}_3(y_4) &= x_3 y_4 - x_2 y_5 + z_7, & \mathcal{L}_3(z_6) &= x_3 z_6 - x_2 z_7,\\
\mathcal{L}_5(x_2) &= z_7, & \mathcal{L}_5(y_4) &= x_3 z_6 - x_2 z_7, & \mathcal{L}_5(z_6) &= y_5 z_6 - y_4 z_7.
\end{align*}

By \eqref{y3}, \eqref{y5} we have
\begin{align*}
w_6 &= 3 x_2 y_4 - \frac{1}{2} y_6 + 3 z_6, \qquad 
w_8 = 3 x_2 z_6 - \frac{1}{2} z_8, \\ 
w_{10} &= \frac{1}{2} (x_4 z_6 - x_3 z_7 + x_2 z_8) - \left(4 x_2^2 + y_4\right) z_6.
\end{align*}

Denote
$p_7 = \mathcal{L}_1(w_6) = \mathcal{L}_3(y_4)$, $p_9 = \mathcal{L}_1(w_8) = \mathcal{L}_3(z_6) = \mathcal{L}_5(y_4)$,
\mbox{$p_{11} = \mathcal{L}_1(w_{10}) = \mathcal{L}_5(z_6)$,}
$w_9 = \mathcal{L}_3(w_6)$, $w_{11} = \mathcal{L}_5(w_6) = \mathcal{L}_3(w_8)$, $w_{13}  = \mathcal{L}_5(w_8) = \mathcal{L}_3(w_{10})$,
$w_{15} = \mathcal{L}_5(w_{10})$. We have
\begin{align*}
p_7 &= x_3 y_4 - x_2 y_5 + z_7, & p_9 &= x_3 z_6 - x_2 z_7, & p_{11} &= y_5 z_6 - y_4 z_7, 
\end{align*}
\begin{align*}
w_9 &= - x_3 x_2 y_4 + x_2^2 y_5 - {1 \over 2} x_4 y_5 + {1 \over 2} x_3 y_6 + y_4 y_5 + x_3 z_6 - 2 x_2 z_7,\\
w_{11} &= - x_2 x_3 z_6 + y_5 z_6 + x_2^2 z_7 - {1 \over 2} x_4 z_7 + {1 \over 2} x_3 z_8, \\
w_{13} &= - y_5 x_2 z_6 + x_2 y_4 z_7 - {1 \over 2} y_6 z_7 + {1 \over 2} y_5 z_8 + z_6 z_7, \\
w_{15} &=  - y_4 y_5 z_6 + y_4^2 z_7 + x_3 z_6^2 - x_2 z_6 z_7 - {1 \over 2} z_7 z_8 + {1 \over 2} z_8 p_7 - {1 \over 2} y_6 p_9 + {1 \over 2} x_4 p_{11}.
\end{align*}

Now for $\mathcal{L}_0, \mathcal{L}_2, \mathcal{L}_4, \mathcal{L}_6, \mathcal{L}_8, \mathcal{L}_{10}$ we give their commutators with $\mathcal{L}_1$:
\begin{equation} \label{com1}
  \begin{pmatrix}
 [\mathcal{L}_1, \mathcal{L}_0] \\
 [\mathcal{L}_1, \mathcal{L}_2] \\
 [\mathcal{L}_1, \mathcal{L}_4] \\
 [\mathcal{L}_1, \mathcal{L}_6] \\
 [\mathcal{L}_1, \mathcal{L}_8] \\
 [\mathcal{L}_1, \mathcal{L}_{10}] 
 \end{pmatrix}
= 
 \begin{pmatrix}
 -1 & 0 & 0\\
 x_2 & -1 & 0\\
 y_4 & x_2 & -1 \\
 z_6 & y_4 & x_2 \\
 0 & z_6 & y_4\\
 0 & 0 & z_6
 \end{pmatrix}
 \begin{pmatrix}
  \mathcal{L}_1 \\ \mathcal{L}_3\\ \mathcal{L}_5
 \end{pmatrix}.
\end{equation}
As the right hand sides of the commutators have already been defined, these relations determine
$\mathcal{L}_2, \mathcal{L}_4, \mathcal{L}_6, \mathcal{L}_8, \mathcal{L}_{10}$ given their values on $x_2, y_4, z_6$.

Let us give these values:
\begin{align*}
\mathcal{L}_2(x_2) &= {12 \over 7} \lambda_{4} + 2 x_2^2 + 4 y_4, \quad
\mathcal{L}_2(y_4) = - {8 \over 7} \lambda_4 x_2 + 2 x_2 y_4 + 6 z_6, \quad
\mathcal{L}_2(z_6) = - {4 \over 7} \lambda_4 y_4 + 2 x_2 z_6, \\
\mathcal{L}_4(x_2) &= {4 \over 7} \lambda_6 - 2 x_2 y_4 + y_6 + 2 z_6, \qquad \mathcal{L}_4(z_6) = - {6 \over 7} \lambda_6 y_4 - 16 x_2^2 z_6  + 3 x_4 z_6 - 8 y_4 z_6 - {1 \over 2} x_3 z_7, \\
\mathcal{L}_4(y_4) &= - {12 \over 7} \lambda_6 x_2 - 10 x_2^2 y_4 + 2 x_4 y_4 - 4 y_4^2 - {1 \over 2} x_3 y_5 + 2 x_2 z_6, \\
\mathcal{L}_6(x_2) &= - {4 \over 7} \lambda_8 + 4 x_2^2 y_4 - x_2 y_6 - 4 x_2 z_6 + {1 \over 2} x_3 y_5 + 2 z_8, \\
\mathcal{L}_6(y_4) &= - {16 \over 7} \lambda_8 x_2 + 2 \lambda_6 y_4 - 2 x_2 y_4^2 - y_5^2 + y_4 y_6 - 16 x_2^2 z_6 + 3 x_4 z_6 - 6 y_4 z_6 - {1 \over 2} x_3 z_7, \\
\mathcal{L}_6(z_6) &= - {8 \over 7} \lambda_8 y_4 + 4 \lambda_6 z_6 - 2 x_2 y_4 z_6 + y_6 z_6 - y_5 z_7 + 2 z_6^2, \\
\mathcal{L}_8(x_2) &= {2 \over 7} \lambda_{10} + x_4 z_6 - 2 y_4 z_6 - x_2 z_8, \\
\mathcal{L}_8(y_4) &= - {6 \over 7} \lambda_{10} x_2 + x_3^2 z_6 - x_2 x_4 z_6 - 18 x_2 y_4 z_6 + 3 y_6 z_6 - x_2 x_3 z_7 - {5 \over 2} y_5 z_7 + x_2^2 z_8 + 2 y_4 z_8 - 6 z_6^2, \\
\mathcal{L}_8(z_6) &= - {10 \over 7} \lambda_{10} y_4 - x_4 y_4 z_6 + {3 \over 2} x_3 y_5 z_6 
+ 2 y_4^2 z_6 - 20 x_2 z_6^2 - z_7^2 + 3 z_6 z_8 + 4 w_6 x_2 z_6 - {1 \over 2} z_7 p_7, \\
\mathcal{L}_{10}(x_2) &= {1 \over 7} \lambda_{12} + {1 \over 2} y_6 z_6 - {1 \over 2} y_4 z_8 - z_6^2, \\
\mathcal{L}_{10}(y_4) &= - {3 \over 7} \lambda_{12} x_2 + {1 \over 2} x_3 y_5 z_6 - {1 \over 2} x_2 y_6 z_6 - {1 \over 2} x_3 y_4 z_7 + {1 \over 2} x_2 y_4 z_8
- 11 x_2 z_6^2 - {3 \over 2} z_7^2 + 3 z_6 z_8, \\
\mathcal{L}_{10}(z_6) &= {2 \over 7} \lambda_{12} y_4 - 4 x_2 y_4^2 z_6 + {1 \over 2} y_5^2 z_6 - y_4 y_5 z_7 + y_4^2 z_8 + 8 x_2^2 z_6^2 + x_2 z_7^2 - 2 x_2 z_6 z_8.
\end{align*}

\begin{lem} \label{l9}
For the polynomial vector fields $\mathcal{L}_0, \mathcal{L}_1, \mathcal{L}_2, \mathcal{L}_3, \mathcal{L}_4, \mathcal{L}_6, \mathcal{L}_5, \mathcal{L}_8, \mathcal{L}_{10}$ 
defined in this section for $k = 4,6,8,10,12,14$ we have
\begin{align*}
 \mathcal{L}_0(\lambda_k) &= L_0(\lambda_k), &
 \mathcal{L}_1(\lambda_k) &= 0, &
 \mathcal{L}_2(\lambda_k) &= L_2(\lambda_k), \\
 \mathcal{L}_4(\lambda_k) &= L_4(\lambda_k),&
 \mathcal{L}_3(\lambda_k) &= 0, &
 \mathcal{L}_6(\lambda_k) &= L_6(\lambda_k),\\
 \mathcal{L}_8(\lambda_k) &= L_8(\lambda_k), &
 \mathcal{L}_5(\lambda_k) &= 0, &
 \mathcal{L}_{10}(\lambda_k) &= L_{10}(\lambda_k).
\end{align*}
\end{lem}

\begin{proof}
 The proof is a straightforward calculation obtained by applying the polynomial vector fields $\mathcal{L}_s$ to the polynomials $\lambda_k$.
\end{proof}

\begin{lem}
 The vector fields $\mathcal{L}_0, \mathcal{L}_1, \mathcal{L}_2, \mathcal{L}_3, \mathcal{L}_4, \mathcal{L}_6, \mathcal{L}_5, \mathcal{L}_8, \mathcal{L}_{10}$ solve Problem \ref{p6}.
\end{lem}

\begin{proof}
 By Lemma \ref{l9} the vector fileds $\mathcal{L}_k$ are projectable for $p$. 
 
We have
$\det \mathcal{T} = - 64 \det T$ for $T$ determined by \eqref{T3}. Therefore, the vector fields $\mathcal{L}_k$ are independent at any point of $p^{-1}(\mathcal{B})$.
\end{proof}

Now let us describe the polynomial Lie algebra for the vector fields
$\mathcal{L}_k$.
The commutators $[\mathcal{L}_0, \mathcal{L}_k]$, $[\mathcal{L}_1, \mathcal{L}_k]$ and $[\mathcal{L}_3, \mathcal{L}_5]$ have been given above.

The commutators of $\mathcal{L}_3$ with $\mathcal{L}_2, \mathcal{L}_4, \mathcal{L}_6, \mathcal{L}_8, \mathcal{L}_{10}$ are
\begin{equation} \label{com3}
  \begin{pmatrix}
 [\mathcal{L}_3, \mathcal{L}_2] \\
 [\mathcal{L}_3, \mathcal{L}_4] \\
 [\mathcal{L}_3, \mathcal{L}_6] \\
 [\mathcal{L}_3, \mathcal{L}_8] \\
 [\mathcal{L}_3, \mathcal{L}_{10}] 
 \end{pmatrix}
= 
 \begin{pmatrix}
 y_4 - \lambda_4 & 0 & - 3\\
 w_6 & y_4 - \lambda_4 & 0 \\
 w_8 & w_6 & y_4 - \lambda_4 \\
 0 & w_8 & w_6\\
 0 & 0 & w_8
 \end{pmatrix}
 \begin{pmatrix}
  \mathcal{L}_1 \\ \mathcal{L}_3\\ \mathcal{L}_5
 \end{pmatrix} +
 {3 \over 7}
 \begin{pmatrix}
 5 \lambda_4 \\
 4 \lambda_6 \\
 3 \lambda_8 \\
 2 \lambda_{10} \\
 \lambda_{12}
 \end{pmatrix}
 \mathcal{L}_1.
\end{equation}
The commutators of $\mathcal{L}_5$ with $\mathcal{L}_2, \mathcal{L}_4, \mathcal{L}_6, \mathcal{L}_8, \mathcal{L}_{10}$ are
\begin{equation} \label{com5}
  \begin{pmatrix}
 [\mathcal{L}_5, \mathcal{L}_2] \\
 [\mathcal{L}_5, \mathcal{L}_4] \\
 [\mathcal{L}_5, \mathcal{L}_6] \\
 [\mathcal{L}_5, \mathcal{L}_8] \\
 [\mathcal{L}_5, \mathcal{L}_{10}] 
 \end{pmatrix}
= 
 \begin{pmatrix}
 z_6 & 0 & 0\\
 w_8 & z_6 & 0 \\
 w_{10} & w_8 & z_6 \\
 - \lambda_{12} & w_{10} & w_8\\
 -2 \lambda_{14} & - \lambda_{12} & w_{10}
 \end{pmatrix}
 \begin{pmatrix}
  \mathcal{L}_1 \\ \mathcal{L}_3\\ \mathcal{L}_5
 \end{pmatrix} +
 {2 \over 7}
 \begin{pmatrix}
 2 \lambda_{4} \\
 3 \lambda_6 \\
 4 \lambda_8 \\
 5 \lambda_{10} \\
 6 \lambda_{12} 
 \end{pmatrix}
 \mathcal{L}_3
 - \begin{pmatrix}
 0 \\
 3 \lambda_4 \\
 2 \lambda_6 \\
 \lambda_8 \\
 0
 \end{pmatrix}
 \mathcal{L}_5.
\end{equation}
The remaining commutators are
\[
 \begin{pmatrix}
 [\mathcal{L}_2, \mathcal{L}_4] \\
 [\mathcal{L}_2, \mathcal{L}_6] \\
 [\mathcal{L}_2, \mathcal{L}_8] \\
 [\mathcal{L}_2, \mathcal{L}_{10}] \\
 [\mathcal{L}_4, \mathcal{L}_6] \\
 [\mathcal{L}_4, \mathcal{L}_8] \\
 [\mathcal{L}_4, \mathcal{L}_{10}] \\
 [\mathcal{L}_6, \mathcal{L}_8] \\
 [\mathcal{L}_6, \mathcal{L}_{10}] \\
 [\mathcal{L}_8, \mathcal{L}_{10}] 
 \end{pmatrix}
=
\mathcal{M}
 \begin{pmatrix}
 \mathcal{L}_0 \\
 \mathcal{L}_2 \\
 \mathcal{L}_4 \\
 \mathcal{L}_6 \\
 \mathcal{L}_8 \\
 \mathcal{L}_{10}
 \end{pmatrix}
+ {1 \over 2}
 \begin{pmatrix}
- y_5 & x_3 & 0 \\
- p_7 - z_7 & y_5 & x_3 \\
- 2 p_9 & z_7 & y_5 \\
- p_{11} & 0 & z_7 \\
 - w_9 & p_7 - 2 z_7 & 2 y_5 \\
 - 2 w_{11} & 0 & 2 p_7 \\
 - w_{13} & - p_{11} & 2 p_9 \\
- 2 w_{13} & 2 p_{11} - w_{11} & w_9 \\
- w_{15} & - w_{13} & w_{11} + p_{11}\\
0 & - w_{15} & w_{13} 
 \end{pmatrix}
 \begin{pmatrix}
 \mathcal{L}_1 \\
 \mathcal{L}_3 \\
 \mathcal{L}_5 \\
 \end{pmatrix},
\]
where $\mathcal{M}$ is given by \eqref{M}.

\section{Representation of the generators in classic form} \label{s10}

\begin{thm} \label{tl}
For genus $g=3$ the generators of the $\mathcal{F}$-module $\Der \mathcal{F}$ are
\begin{align*}
\mathscr{L}_1 &= \partial_{u_1}, \qquad \mathscr{L}_3 = \partial_{u_3}, \qquad \mathscr{L}_5 = \partial_{u_5}, \\
\mathscr{L}_0 &= L_0 - u_1 \partial_{u_1} - 3 u_3 \partial_{u_3} - 5 u_5 \partial_{u_5}, \\
\mathscr{L}_2 &= L_2 - \left(\zeta_1 - {8 \over 7} \lambda_4 u_3 \right) \partial_{u_1}
- \left( u_1 - {4 \over 7} \lambda_4 u_5 \right) \partial_{u_3} - 3 u_3 \partial_{u_5}, \\
\mathscr{L}_4 &= L_4 - \left(\zeta_3 - {12 \over 7} \lambda_6 u_3 \right) \partial_{u_1}
- \left( \zeta_1 + \lambda_4 u_3 - {6 \over 7} \lambda_6 u_5 \right) \partial_{u_3} - (u_1 + 3 \lambda_4 u_5) \partial_{u_5}, \\
\mathscr{L}_6 &= L_6 - \left(\zeta_5 - {9 \over 7} \lambda_8 u_3 \right) \partial_{u_1}
- \left(\zeta_3 - {8 \over 7} \lambda_8 u_5 \right) \partial_{u_3}
- \left(\zeta_1 + \lambda_4 u_3 + 2 \lambda_6 u_5 \right) \partial_{u_5}, \\
\mathscr{L}_8 &= L_8 + \left({6 \over 7} \lambda_{10} u_3 - \lambda_{12} u_5\right) \partial_{u_1}
- \left(\zeta_5 - {10 \over 7} \lambda_{10} u_5 \right) \partial_{u_3}
- \left(\zeta_3 + \lambda_8 u_5 \right) \partial_{u_5}, \\
\mathscr{L}_{10} &= L_{10} + \left( {3 \over 7} \lambda_{12} u_3 - 2 \lambda_{14} u_5 \right) \partial_{u_1}
+ {5 \over 7} \lambda_{12} u_5 \partial_{u_3} - \zeta_5 \partial_{u_5}, 
\end{align*}
where the vector fields $L_{2k}$ are given explicitly by \eqref{Lk}.
\end{thm}

\begin{proof}
To obtain $\mathscr{L}_k f \in \mathcal{F}$ for $f \in \mathcal{F}$ we take $\mathscr{L}_k$ such that
$
\mathscr{L}_k \varphi^* x_{i,j} = \varphi^* \mathcal{L}_k x_{i,j}
$ for the coordinate functions $x_{i,j}$ in $\mathbb{C}^{3g}$.

We have $\mathscr{L}_1 = \partial_{u_1}$, $\mathscr{L}_3 = \partial_{u_3}$, $\mathscr{L}_5 = \partial_{u_5}$, see Section \ref{s2}.
The vector field
$\mathscr{L}_0$ is the Euler vector field and $\mathscr{L}_0 \in \Der \mathcal{F}$ follows from homogeneity properties of hyperelliptic sigma function.
The vector fields $\mathscr{L}_2, \mathscr{L}_4, \mathscr{L}_6, \mathscr{L}_8, \mathscr{L}_{10}$ are determined by the conditions:
\begin{enumerate}
\item $\mathscr{L}_{2k} = L_{2k} + f_{2k,1}(u, \lambda) \partial_{u_1} + f_{2k,3}(u, \lambda) \partial_{u_3} + f_{2k,5}(u, \lambda) \partial_{u_5}$,
\item the fields $\mathscr{L}_k$ satisfy commutation relations obtained by $\varphi^*$ from \eqref{com1}, \eqref{com3}, \eqref{com5}.
\end{enumerate}
The condition (2) determines the coefficients $f_{2k,j}(u, \lambda)$, $j = 1,3,5$, up to constants \mbox{in $u_1$, $u_3$, $u_5$.}
The grading of the variables and the condition $[\mathscr{L}_0, \mathscr{L}_k] = k \mathscr{L}_k$ determines these constants.
\end{proof}

\begin{cor}
The Lie algebra for the generators from Theorem \ref{tl} of the $\mathcal{F}$-module $\Der \mathcal{F}$ in the genus $g=3$ case is
\begin{align*}
[\mathscr{L}_0, \mathscr{L}_k] &= k \mathscr{L}_k, & [\mathscr{L}_1, \mathscr{L}_3] &= 0, & [\mathscr{L}_1, \mathscr{L}_5] &= 0, & [\mathscr{L}_3, \mathscr{L}_5] &= 0,
\end{align*}
\[
  \begin{pmatrix}
 [\mathscr{L}_1, \mathscr{L}_2] \\
 [\mathscr{L}_1, \mathscr{L}_4] \\
 [\mathscr{L}_1, \mathscr{L}_6] \\
 [\mathscr{L}_1, \mathscr{L}_8] \\
 [\mathscr{L}_1, \mathscr{L}_{10}] 
 \end{pmatrix}
= 
 \begin{pmatrix}
 \wp_2 & -1 & 0\\
 \wp_{1;3} & \wp_2 & -1 \\
 \wp_{1;5} & \wp_{1;3} & \wp_2 \\
 0 & \wp_{1;5} & \wp_{1;3} \\
 0 & 0 & \wp_{1;5}
 \end{pmatrix}
 \begin{pmatrix}
  \mathscr{L}_1 \\ \mathscr{L}_3\\ \mathscr{L}_5
 \end{pmatrix},
\]
\[
  \begin{pmatrix}
 [\mathscr{L}_3, \mathscr{L}_2] \\
 [\mathscr{L}_3, \mathscr{L}_4] \\
 [\mathscr{L}_3, \mathscr{L}_6] \\
 [\mathscr{L}_3, \mathscr{L}_8] \\
 [\mathscr{L}_3, \mathscr{L}_{10}] 
 \end{pmatrix}
= 
 \begin{pmatrix}
 \wp_{1;3} - \lambda_4 & 0 & - 3\\
 \wp_{0;3,3} & \wp_{1;3} - \lambda_4 & 0 \\
 \wp_{0;3,5} & \wp_{0;3,3} & \wp_{1;3} - \lambda_4 \\
 0 & \wp_{0;3,5} & \wp_{0;3,3}\\
 0 & 0 & \wp_{0;3,5}
 \end{pmatrix}
 \begin{pmatrix}
  \mathscr{L}_1 \\ \mathscr{L}_3\\ \mathscr{L}_5
 \end{pmatrix} +
 {3 \over 7}
 \begin{pmatrix}
 5 \lambda_4 \\
 4 \lambda_6 \\
 3 \lambda_8 \\
 2 \lambda_{10} \\
 \lambda_{12}
 \end{pmatrix}
 \mathscr{L}_1,
\]
\[
  \begin{pmatrix}
 [\mathscr{L}_5, \mathscr{L}_2] \\
 [\mathscr{L}_5, \mathscr{L}_4] \\
 [\mathscr{L}_5, \mathscr{L}_6] \\
 [\mathscr{L}_5, \mathscr{L}_8] \\
 [\mathscr{L}_5, \mathscr{L}_{10}] 
 \end{pmatrix}
= 
 \begin{pmatrix}
 \wp_{1;5} & 0 & 0\\
 \wp_{0;3,5} & \wp_{1;5} & 0 \\
 \wp_{0;5,5} & \wp_{0;3,5} & \wp_{1;5} \\
 - \lambda_{12} & \wp_{0;5,5} & \wp_{0;3,5}\\
 -2 \lambda_{14} & - \lambda_{12} & \wp_{0;5,5}
 \end{pmatrix}
 \begin{pmatrix}
  \mathscr{L}_1 \\ \mathscr{L}_3\\ \mathscr{L}_5
 \end{pmatrix} +
 {2 \over 7}
 \begin{pmatrix}
 2 \lambda_{4} \\
 3 \lambda_6 \\
 4 \lambda_8 \\
 5 \lambda_{10} \\
 6 \lambda_{12} 
 \end{pmatrix}
 \mathscr{L}_3
 - \begin{pmatrix}
 0 \\
 3 \lambda_4 \\
 2 \lambda_6 \\
 \lambda_8 \\
 0
 \end{pmatrix}
 \mathscr{L}_5,
\]
\[
 \begin{pmatrix}
 [\mathscr{L}_2, \mathscr{L}_4] \\
 [\mathscr{L}_2, \mathscr{L}_6] \\
 [\mathscr{L}_2, \mathscr{L}_8] \\
 [\mathscr{L}_2, \mathscr{L}_{10}] \\
 [\mathscr{L}_4, \mathscr{L}_6] \\
 [\mathscr{L}_4, \mathscr{L}_8] \\
 [\mathscr{L}_4, \mathscr{L}_{10}] \\
 [\mathscr{L}_6, \mathscr{L}_8] \\
 [\mathscr{L}_6, \mathscr{L}_{10}] \\
 [\mathscr{L}_8, \mathscr{L}_{10}] 
 \end{pmatrix}
=
\mathcal{M}
 \begin{pmatrix}
 \mathscr{L}_0 \\
 \mathscr{L}_2 \\
 \mathscr{L}_4 \\
 \mathscr{L}_6 \\
 \mathscr{L}_8 \\
 \mathscr{L}_{10}
 \end{pmatrix}
+ {1 \over 2}
 \begin{pmatrix}
- \wp_{2;3} & \wp_{3} & 0 \\
- \wp_{1;3,3} - \wp_{2;5} & \wp_{2;3} & \wp_{3} \\
- 2 \wp_{1;3,5} & \wp_{2;5} & \wp_{2;3} \\
- \wp_{1;5,5} & 0 & \wp_{2;5} \\
 - \wp_{0;3,3,3} & \wp_{1;3,3} - 2 \wp_{2;5} & 2 \wp_{2;3} \\
 - 2 \wp_{0;3,3,5} & 0 & 2 \wp_{1;3,3} \\
 - \wp_{0;3,5,5} & - \wp_{1;5,5} & 2 \wp_{1;3,5} \\
- 2 \wp_{0;3,5,5} & 2 \wp_{1;5,5} - \wp_{0;3,3,5} & \wp_{0;3,3,3} \\
- \wp_{0;5,5,5} & - \wp_{0;3,5,5} & \wp_{0;3,3,5} + \wp_{1;5,5}\\
0 & - \wp_{0;5,5,5} & \wp_{0;3,5,5} 
 \end{pmatrix}
 \begin{pmatrix}
 \mathscr{L}_1 \\
 \mathscr{L}_3 \\
 \mathscr{L}_5 \\
 \end{pmatrix},
\]
where $\mathcal{M}$ is given by \eqref{M}.
\end{cor}

\begin{proof}[The proof]
is obtained by applying $\varphi^*$ to the polynomial Lie algebra in Section~\ref{s9}.
\end{proof}

\end{document}